\documentclass{birkjour}
\usepackage{amssymb}
\usepackage{enumerate}
\usepackage{epstopdf}
\allowdisplaybreaks

\usepackage{graphicx,hyperref}

\newtheorem{theorem}{Theorem}[section]

\newtheorem{lemma}[theorem]{Lemma}

\newtheorem{definition}[theorem]{Definition}

\def\ds{\displaystyle}

\title[Leslie-Gower model with free boundary] 
      {A Leslie-Gower predator-prey model with a free boundary}

\author[Y. Liu]{\normalsize Yunfeng Liu $^{\small 1}$}

\author[Z. Guo]{\normalsize Zhiming Guo $^{1}$}
\email{guozm@gzhu.edu.cn}
\author[M. El Smaily]{\normalsize Mohammad El Smaily $^{2}$}
\email{m.elsmaily@unb.ca}

\author[L. Wang]{\normalsize Lin Wang $^{2}$}
\email{lwang2@unb.ca}

\address{\\{\small $^{1}$  \sc School of Mathematics and Information Sciences,}\br
{\small \sc Guangzhou University, \br Guangzhou, 510006, PR China}\\
\\
{\small $^{2}$ \sc Department of Mathematics \& Statistics, \br University of New Brunswick,\br Fredericton, NB, E3B5A3, Canada}
}

\subjclass{Primary: 35K57, 35R35, 92C50, 92B99.}
 \keywords{Free Boundary; Leslie-Gower; Biological Invasive, Predator-Prey; Spreading-Vanishing Dichotomy.}

\thanks{The work of ME and LW was supported by  NSERC Discovery Grants from the Natural Sciences and Engineering Research Council of Canada (NSERC).  YL and ZG acknowledge support from the National Natural Science Foundation of China (No.11771104), Program for Changjiang Scholars and Innovative Research Team in University (IRT-16R16).YL was supported by the National Natural Science Foundation of China under Grant No.11271093 and the Innovation Research for the Postgraduates of Guangzhou University under Grant No.2017GDJC-D05.}

\thanks{$^*$ Corresponding author: Zhiming Guo}

\begin{document}

\maketitle
\centerline{\scshape Yunfeng Liu $^{a,b}$, Zhiming Guo $^{a,*}$, Mohammad El Smaily $^{b}$ and Lin Wang $^{b}$}
\medskip
{\footnotesize
 \centerline{$^{a}$ School of Mathematics and Information Science,}   \centerline{Guangzhou University}
   \centerline{ Guangzhou, 510006, PR China}
\centerline{$^{b}$ Department of Mathematics and Statistics,}
   \centerline{University of New Brunswick}
   \centerline{ Fredericton, NB, Canada}



\bigskip


\begin{abstract}
In this paper, we consider a Leslie-Gower predator-prey model in one-dimensional environment. We study the asymptotic behavior of two species evolving in a domain with a free boundary. Sufficient conditions for spreading success and spreading failure are obtained. We also derive sharp criteria for spreading and vanishing of the two species. Finally, when spreading is successful, we show that the spreading speed is between the minimal speed of traveling wavefront solutions for the predator-prey model on the whole real line (without a free boundary) and an elliptic problem that follows from the original model.
\end{abstract}
\maketitle

\section{Introduction}

A variety of models are used to describe the predator-prey interactions. The dynamical relationship between a predator and a prey has long been among the dominant topics in mathematical ecology due to its universal existence and importance. Recently, many works studied the predator-prey system with the Leslie-Gower scheme \cite{MM2003,L-G4,HH1995,L-G1,WN2016,L-G5,ZhouJun2014}. A typical Leslie-Gower predator-prey model is the following
\begin{equation}\label{1.1}
\left\{\begin{array}{l}
\ds {\frac{dN}{dt}=rN\left(1-\frac{N}{G}\right)-bNP,}\vspace{7 pt}\\
\ds{\frac{dP}{dt}=P\left(a-\frac{cP}{N+G_{1}}\right),}
\end{array}\right.
\end{equation}
where $N$ and $P$ denote the population densities of the prey and  predator populations respectively. The parameter $r$ represents the  intrinsic growth rate of the prey species and $G$ stands for its carrying capacity. The parameter $a$ is the growth rate for the predator  and $b$ (resp. $c$) is the maximum value which per capita reduction rate of $N$ (resp. $P$) can attain. $G_{1}$ denotes the extent to which environment provides protection to predator $P$. {\em All parameters are assumed to be positive.}

In order to get the spatiotemporal dynamics of system \eqref{1.1}, the following reaction-diffusion equations are widely accepted
\begin{equation}\label{1.2}
\left\{\begin{array}{ll}
\ds{\frac{\partial N}{\partial t}=d_{1}\frac{\partial^{2} N}{\partial x^{2}} + rN\left(1-\frac{N}{G}\right)-bNP,}~~& t,~x \in \mathbb{R},\vspace{7 pt}\\
\ds{\frac{\partial P}{\partial t}=d_{2}\frac{\partial^{2} P}{\partial x^{2}}+P\left(a-\frac{cP}{N+G_{1}}\right)},\qquad\qquad\qquad&~t,~x \in \mathbb{R}.
\end{array}\right.
\end{equation}
By setting
  $$N=Gu,~P=\frac{aG}{c}\upsilon,~t=\frac{\hat{t}}{r},~x=\sqrt{\frac{d_{1}}{r}}\hat{x},$$
  $$\delta=\frac{abG}{rc},~\alpha=\frac{G_{1}}{G},~\kappa=\frac{a}{r}\text{ and }
  D=\frac{d_{2}}{d_{1}},$$
and dropping the hat sign, \eqref{1.2} turns into the following system
\begin{equation}\label{1.3}
\left\{\begin{array}{ll}
\ds{\frac{ \partial u}{ \partial t}=u_{xx} + u(1-u)-\delta u\upsilon,} &t, x \in \mathbb{R},\vspace{7 pt}\\
\ds{\frac{ \partial \upsilon}{ \partial t}=D\upsilon_{xx}+\kappa\upsilon\left(1-\frac{\upsilon}{u+\alpha}\right),}\qquad\qquad\qquad&t, x \in \mathbb{R} .\
\end{array}\right.
\end{equation}
  System \eqref{1.3} has at least three boundary equilibrium solutions $E_{1}=(0,0),$  $E_{2}=(0,\alpha),$  $E_{3}=(1,0)$. Moreover, if $ \delta\alpha<1,$    there exists a unique interior equilibrium solution $E_{*}=(u^{*}, \upsilon^{*}),$ where $$\ds{\upsilon^{*}=\alpha+u^{*}} \text{ and }\ds{u^{*}=\frac{1-\delta\alpha}{1+\delta}}.$$

Our main objective is to understand the long time behavior of a Leslie-Gower predator-prey model via a free boundary. In this paper, we  consider the following model:
\begin{equation}\label{FBP}
\left\{\begin{array}{ll}
\ds{\frac{\partial u}{\partial t}=u_{xx} + u(1-u)-\delta u\upsilon,}&\text{for all } t>0\text{ and }0<x<h(t),\vspace{7 pt}\\
\ds{\frac{\partial \upsilon}{\partial t}=D\upsilon_{xx}+\kappa\upsilon\left(1-\frac{\upsilon}{u+\alpha}\right),}&\text{for all }t>0\text{ and }0<x<h(t), \vspace{7 pt}\\
h'(t)=-\mu (u_{x}(t,h(t))+\rho \upsilon_{x}(t,h(t))),&\text{for all }t>0,\vspace{7 pt}\\
h(0)=h_{0},\vspace{7 pt}\\
u_{x}(t,0)=\upsilon_{x}(t,0)=u(t,h(t))=\upsilon(t,h(t))=0,~&\text{for all }t>0,\vspace{7 pt}\\
u(0,x)=u_{0}(x)~\text{ and }~\upsilon(0,x)=\upsilon_{0}(x),\qquad& \text{for all }x\in[0,h_{0}],
\end{array}\right.
\end{equation}
with the positive parameters $\mu,$  $\rho > 0.$ The initial data $(u_{0},\upsilon_{0})$ satisfy
\begin{equation}\label{1.5}
\left\{\begin{array}{l}
u_{0},~\upsilon_{0}\in C^{2}([0,h_{0}]),\vspace{7 pt}\\
u'_{0}(0)=\upsilon'_{0}(0)=u_{0}(h_{0})=\upsilon_{0}(h_{0})=0,\vspace{7 pt}\\
h_{0}>0,~u_{0}(x)>0 ~ \text{ and }~\upsilon_{0}(x)>0~ \text{ for all }~ x\in [0,h_{0}).
\end{array}\right.
\end{equation}

From a biological point of view, model \eqref{FBP} describes how the two species evolve if they initially occupy the bounded region $[0, h_{0}]$. The homogeneous Neumann boundary condition at $x=0$ indicates that the left boundary is fixed, with the population confined to move only to right of the boundary point $x=0.$ We assume that both species have a tendency to emigrate throught the right boundary point to obtain their new habitat: the free boundary $x=h(t)$ represents the spreading front. Moreover, it is assumed that the expanding speed of the free boundary is proportional to the normalized population gradient at the free boundary. This is well-known as the Stefan condition.

Many previous works study free boundary problems  in predator-prey models. We refer the reader, for instance, to \cite{ziyou1,ziyou2,ziyou3,ziyou4} and references cited therein.

\medskip
In this paper, we have been working under the following assumption
$$\text{\bf{(H1)}}:\quad \delta\alpha+\delta<1.$$

\paragraph*{Organization of the paper.}  In Section 2, we use a contraction mapping argument to prove the local existence and uniqueness of the solution to \eqref{FBP}, then make use of suitable estimates on the solution to show that it exists for all time $t>0$. In Section 3, we derive several lemmas which will be used later. Section \ref{dichotomy} is devoted to the long time behavior of $(u,\upsilon),$ proving a spreading-vanishing dichotomy and finally deriving criteria for spreading and vanishing. We estimate the spreading speed  in Section \ref{spreading.estimate} and then summarize through a brief discussion in Section \ref{discuss}.

\section{Existence and uniqueness of solutions}
 In this section, we first state a result about the local existence and uniqueness of a solution to \eqref{FBP} in Lemma \ref{local.exist}. Then we derive a priori  estimates (Lemma \ref{estimates}) in order justify that the solution is defined for all time $t>0$. The global existence of a solution to the system \eqref{FBP} is stated in Theorem \ref{th5.1}.
\begin{lemma}\label{local.exist}
Assume that $(u_{0},\upsilon_{0})$ satisfies the condition \eqref{1.5}, then for any $\theta\in (0,1),$  there is a $T> 0$ such that the problem~(\ref{FBP}) admits a unique solution $(u(t,x), \upsilon(t,x), h(t)),$  which satisfies
$$(u,\upsilon,h) \in C^{\frac{(1+\theta)}{2},1+\theta}(Q_{T})\times C^{\frac{(1+\theta)}{2},1+\theta}(Q_{T})\times C^{1+\frac{\theta}{2}}([0,T]).$$
where $Q_{T}=\{(t,x)\in \mathbb{R}^{2}:~t\in[0,T],~x\in[0,h(t)]\}$.
\end{lemma}
The proof of Lemma  \ref{local.exist} will be postponed to Section \ref{exist.}.

\begin{lemma}\label{estimates}
Let $(u,\upsilon,h(t))$ be a solution of \eqref{FBP} for $t\in [0,T]$ \text{for some} $T>0$. Then
\begin{equation}\label{5.6}
0<u(t,x)\leq \max\{1, \|u_{0}\|_{\infty}\}:=M_{1}~\text{ for }~t\in [0,T] \text{ and }x\in [0,h(t)),
\end{equation}
\begin{equation}\label{5.7}
0< \upsilon(t,x)\leq \max\{M_{1}+\alpha, \|\upsilon_{0}\|_{\infty}\}:=M_{2}~\text{ for }~t\in [0,T]\text{ and } x\in [0,h(t)),
\end{equation}
\begin{equation}\label{5.8}
0< h'(t)\leq \Lambda~\text{ for all }~t\in (0,T].
\end{equation}
where $\Lambda> 0$  depends on $\mu,$  $\rho,$  $D,$  $\kappa,$  $\|u_{0}\|_{\infty},$  $\|\upsilon_{0}\|_{\infty},$  $\|u'\|_{C[0,h_{0}]}$ and $\|\upsilon'\|_{C[0,h_{0}]}$.
\end{lemma}

The proof of Lemma  \ref{estimates} will be postponed to Section \ref{exist.} as well.

\begin{theorem}\label{th5.1}
Assume that $(u_{0},\upsilon_{0})$ satisfies the condition \eqref{1.5}, then for any $\theta\in (0,1),$  the problem~(\ref{FBP}) admits a unique solution $(u(t,x), \upsilon(t,x), h(t)),$  which satisfies
$$(u,\upsilon,h) \in C^{\frac{(1+\theta)}{2},1+\theta}(Q)\times C^{\frac{(1+\theta)}{2},1+\theta}(Q)\times C^{1+\frac{\theta}{2}}([0,+\infty)),$$
where \[Q=\{(t,x)\in \mathbb{R}^{2}: t\in[0,+\infty),~x\in[0,h(t)]\}.\]
\end{theorem}
\begin{proof}[On the proof of Theorem \ref{th5.1}.]  We only give a brief sketch of the proof here since it is similar to those done in  \cite{dushou} and \cite{dusupandinf2014}: the global existence of the solution to problem \eqref{FBP} follows  from the uniqueness of the local solution, Zorn's lemma and the uniform estimates of $u,$  $\upsilon$ and  $h'(t)$ obtained in Lemma \ref{estimates}, above.\end{proof}
\section{Known results from prior works}
In this section, we recall from prior works some important results that will be used repeatedly in our arguments. We start with some results regarding the stationary state(s) of the  model
\begin{equation}\label{2.1}
\left\{\begin{array}{ll}
\ds{\frac{\partial u}{\partial t}=d u_{xx}+a u(1-bu),}\qquad &(t,x)\in (0, \infty)\times (0, L),\vspace{7 pt}\\
u_{x}(t,0)=u(t,L)=0,  &  t>0.
\end{array}\right.
\end{equation}
The  stationary state will be determined via  the eigenvalue problem
\begin{equation}\label{2.11}
\left\{\begin{array}{ll}
d\phi_{xx}+a\phi=\sigma\phi, \qquad &0< x< L,\vspace{7 pt}\\
\phi_{x}(0)=\phi(L)=0&
\end{array}\right.
\end{equation}
as well as the spatial domain's size.
The following lemma summarizes the result.
\begin{lemma}[\cite{Cantrell2003} and \cite{zhouxiao}]\label{le2.1}
Let $L^{*}=\ds{\frac{\pi}{2}\sqrt{\frac{d}{a}}}$ and $\ds{d^{*}=\frac{4aL^{2}}{\pi^{2}}.}$ Then we have:
\begin{enumerate}[\bf (i)]
\item  if $L\leq L^{*}$, all positive solutions of \eqref{2.1} tend to zero in $C([0,L])$ as $t\rightarrow +\infty.$\\
\item If $L> L^{*}$, then \eqref{2.1} has a minimal positive equilibrium $\phi,$   and all positive solutions to \eqref{2.1} approach $\phi$ in $C([0,L])$ as $t\rightarrow +\infty$.\\
\item If $0< d < d^{*},$  the principal eigenvalue of  \eqref{2.11} is positive ($\sigma_{1}> 0.$) If $d=d^{*}$ then $\sigma_{1}=0,$  and if $d > d^{*}$ then $\sigma_{1}< 0$.
\end{enumerate}
\end{lemma}
For a detailed  proof of (i) and (ii) one can refer to Proposition 3.1 and 3.2 of \cite{Cantrell2003}. The result in (iii) is obtained through a simple computation and can be found in the proof of  Corollary 3.1 in \cite{zhouxiao}.

\medskip

Now, we state a  comparison principle that we will use in the proving the results of Section \ref{dichotomy}, below. This comparison principle is extracted from Lemma 4.1 and Lemma 4.2 of \cite{w3} with minor modifications.
\begin{lemma}\label{le2.2}
Let $\bar{h}$ and  $\underline{h}$ be two  postive $C^{1}([0,+\infty))$  functions {\rm (}$\bar{h}, \underline{h}>0$ in $[0,+\infty)${\rm ).} Denote by
\[\Omega=\left\{(t,x):t>0,x\in[0,\bar{h}(t)]\right\}\] and \[\Omega_{1}=\{(t,x):t>0,x\in[0,\underline{h}(t)]\}.\]
Let $\bar{u}, \bar{\upsilon}\in C(\bar{\Omega})\cap C^{1,2}(\Omega)$ and $\underline{u},\underline{\upsilon}\in C(\bar{\Omega}_{1})\cap C^{1,2}(\Omega_{1}).$ Assume that \[0<\bar{u},~\underline{u}\leq M_{1}\text{ and }0<\bar{\upsilon},\underline{\upsilon}\leq M_{2}\] and  that $(\bar{u},\bar{\upsilon},\bar{h})$ satisfies
\begin{equation}\label{2.2}
\left\{\begin{array}{ll}
\bar{u}_{t}-\bar{u}_{xx}\geq\bar{u}(1-\bar{u}),~\qquad &t>0,0<x<\bar{h}(t),\vspace{7 pt}\\
\bar{\upsilon}_{t}-D\bar{\upsilon}_{xx}\geq \kappa\bar{\upsilon}\left(1-\frac{\bar{\upsilon}}{M_{1}+\alpha}\right),~&t>0,0<x<\bar{h}(t),\vspace{7 pt}\\
\bar{u}_{x}(t,0)\leq 0,\bar{\upsilon}_{x}(t,0)\leq 0,~&t>0,\vspace{7 pt}\\
\bar{u}(t,\bar{h}(t))=\bar{\upsilon}(t,\bar{h}(t))=0,~&t>0,\vspace{7 pt}\\
\bar{h}'(t)\geq -\mu(\bar{u}_{x}(t,\bar{h}(t))+\rho\bar{\upsilon}_{x}(t,\bar{h}(t))),~&t>0,
\end{array}\right.
\end{equation}
and  the couple $(\underline{u}, \underline{h})$ satisfies
\begin{equation}\label{2.20}
\left\{\begin{array}{ll}
\underline{u}_{t}-\underline{u}_{xx}\leq\underline{u}(1-\delta M_{2}-\underline{u}),
~~\qquad \qquad &t>0,0<x<\underline{h}(t),\\
\underline{u}_{x}(t,0)\geq 0,~&t>0,\vspace{7 pt}\\
\underline{u}(t,\underline{h}(t))=0,~&t>0,\vspace{7 pt}\\
\underline{h}'(t)\leq-\mu\underline{u}_{x}(t,\underline{h}(t)),~&t>0
\end{array}\right.
\end{equation}
and the couple $(\underline{\upsilon},\underline{h})$ satisfies
\begin{equation}\label{2.21}
\left\{\begin{array}{ll}
\underline{\upsilon}_{t}-D\underline{\upsilon}_{xx}\leq \kappa\underline{\upsilon}(1-\frac{\underline{\upsilon}}{\alpha}),~\qquad \qquad\qquad &t>0 \quad 0<x<\underline{h}(t),\vspace{7 pt}\\
\underline{\upsilon}_{x}(t,0)\geq 0,~&t>0,\vspace{7 pt}\\
\underline{\upsilon}(t,\underline{h}(t))=0,~&t>0,\vspace{7 pt}\\
\underline{h}'(t)\leq-\mu\rho\underline{\upsilon}_{x}(t,\underline{h}(t)),~&t>0.
\end{array}\right.
\end{equation}
Assume that the initial data of \eqref{2.2} satisfy \[\bar{h}(0)\geq h_{0},~ \bar{u}(0,x), \bar{\upsilon}(0,x)\geq 0\text{ on }[0,\bar{h}(0)]\] and \[\bar{u}(0,x)\geq u_{0}(x)\text{ and }\bar{\upsilon}(0,x)\geq\upsilon_{0}(x)~on~[0,h_{0}],\]
and the initial data  of \eqref{2.20} and \eqref{2.21} satisfy \[\underline{h}(0)\leq h_{0}, ~0<\underline{u}(0,x)\leq u_{0}(x)\text{ and } 0<\underline{\upsilon}(0,x)\leq \upsilon_{0}(x)\text{ on }[0,\underline{h}(0)].\]
Then, the solution $(u,\upsilon,h)$ of \eqref{FBP} satisfies \[\underline{h}(t)\leq h(t)\leq \bar{h}(t)~on~[0,+\infty),\] \[u\leq\bar{u}~\&~\upsilon\leq\bar{\upsilon}\text{ for all }t\geq 0\text{ and } 0\leq x\leq h(t),\] and \[u\geq\underline{u}~\& ~\upsilon\geq\underline{\upsilon}\text{ for all  }t\geq 0 \text{ and }0\leq x\leq \underline{h}(t).\]
\end{lemma}
\noindent The proof of Lemma \ref{le2.2} is very similar to the proofs of Lemma 5.1 of \cite{b}, Lemma 4.1 and Lemma 4.2 of \cite{w3}. We hence omit the details here.

\medskip
In order to discuss the spreading of the species, we will use Lemma A.2, Lemma A.3 of \cite{wx1} and Proposition 8.1 of \cite{w2}. We restate these results here for the reader's convenience.
\begin{lemma}\label{le2.3}
Let $M\geq 0.$ For any given $\varepsilon>0$ and $l_{\varepsilon}>0,$  there exist $\ds{l>\max\left\{l_{\varepsilon}, \frac{\pi}{2}\sqrt{\frac{d}{a}}\right\}}$ such that, if the continuous and non-negative function $U(t,x)$ satisfies
\begin{equation}\label{2.3}
\left\{\begin{array}{ll}
U_{t}-dU_{xx}\geq U(a-bU), \qquad &t>0, 0<x<l,\vspace{7 pt}\\
U_{x}(t,0)=0,\quad U(t,l)\geq M,& t>0,
\end{array}\right.
\end{equation}
and if $U(0,x)>0$ in $[0,l),$  then $$\ds{\liminf_{t\rightarrow+\infty}U(t,x)>\frac{a}{b}-
\varepsilon\text{ uniformly on }
[0,l_{\varepsilon}].}$$
\end{lemma}
\begin{lemma}\label{le2.4}
Let M be a nonnegative constant. For any given $\varepsilon>0$ and $l_{\varepsilon}>0,$  there exists $\ds{l>\max\left\{l_{\varepsilon}, \frac{\pi}{2}\sqrt{\frac{d}{a}}\right\}}$  such that , if the continuous and non-negative function $V(t,x)$ satisfies
\begin{equation}\label{2.4}
\left\{\begin{array}{ll}
V_{t}-dV_{xx}\leq V(a-bV), \qquad &t>0, 0<x<l,\vspace{7 pt}\\
V_{x}(t,0)=0,V(t,l)\leq M, &t>0,
\end{array}\right.
\end{equation}
and if $V(0,x)>0$ in $[0,l),$  then $$ \ds{\limsup_{t\rightarrow+\infty}V(t,x)<\frac{a}{b}+\varepsilon
~\text{uniformly~on}~[0,l_{\varepsilon}].}$$
\end{lemma}

On the contrary, we will use the following lemma, which is Proposition 3.1 of \cite{w3}, in order to discuss the vanishing case of the species.
\begin{lemma}[Proposition 3.1 in \cite{w3}]\label{le2.5}
Let $d$ and $s_{0}$ be positive constants and let $a\in\mathbb{R}$. Assume that $\omega_{0}\in C^{2}([0,s_{0}])$ satisfies \[\omega_{0}'(0)=0,~\omega_{0}(s_{0})=0\text{  and }\omega_{0}(x)>0 \text{ for all }x\in(0,s_{0}).\] Let $s\in C^{1+\frac{\theta}{2}}([0,+\infty))$ and  $\omega\in C^{\frac{1+\theta}{2},1+\theta}([0,\infty)\times[0,s(t)]),$ for some $\theta > 0.$ Assume that  $s(t)>0$ and $\omega(t,x)>0$ for all $0\leq t<\infty$ and $0<x<s(t).$ We further assume that \[\ds{\lim_{t\rightarrow+\infty}s(t)=s_{\infty}<+\infty}, \quad \ds{\lim_{t\rightarrow+\infty}s'(t)=0} \text{ and }\|\omega(t,\cdot)\|_{C^{1}[0,s(t)]}\leq \widetilde{M} \text{ for all }t>1,\] for some constant $\widetilde{M}>0$. If the functions $\omega$ and $s$ satisfy
\begin{equation}\label{2.5}
\left\{\begin{array}{ll}
\omega_{t}-d\omega_{xx} \geq \omega(a-\omega),\qquad\qquad\qquad& t>0\text{ and } 0<x<s(t)\vspace{6 pt},\\
s'(t)\geq-\mu \omega_{x}(t,s(t)), & t>0,\vspace{6 pt}\\
s(0)=s_{0},\\
\omega_{x}(t,0)=0,\quad \omega(t,s(t))=0, & t>0,\\

\omega(0,x)=\omega_{0}(x), & x\in[0,s_{0}],
\end{array}\right.
\end{equation}
then  $$\lim_{t\rightarrow+\infty}\parallel \omega(t,\cdot) \parallel_{C[0,s(t)]}=0.$$
\end{lemma}

To discuss the asymptotic behaviors of $u$ and  $\upsilon$ in the vanishing case, we need the following lemma.
\begin{lemma}\label{le2.8}
Let $(u, \upsilon, h(t))$ be the solution of \eqref{FBP} and recall that $\ds{h_{\infty}=\lim_{t\rightarrow +\infty}h(t)}.$ If $h_{\infty}<\infty,$ then there exists $M,$  for all $t>0,$  such that
$\|u(t,\cdot)\|_{C^{1}[0,h(t)]}\leq M$ and $\|\upsilon(t,\cdot)\|_{C^{1}[0,h(t)]}\leq M$.
Moreover, $\ds{\lim_{t\rightarrow +\infty}h'(t)=0}$.
\end{lemma}
We skip the proof of the above lemma since it is  similar to that of Theorem 4.1 in \cite{w2}.

Furthermore, we need the following lemma which appears in \cite{b}  and \cite{w3} (page 893 and page 3388 respectively).
\begin{lemma}\label{le2.6}
Consider the following problem
\begin{equation}\label{2.6}
\left\{\begin{array}{ll}
\ds{\frac{\partial u}{ \partial t}=u_{xx} + u(1-u),}~\qquad \qquad \qquad&t>0, x>0,\vspace{7 pt}\\
\ds{\frac{\partial \upsilon}{\partial t}=D\upsilon_{xx}+\kappa\upsilon\left(1-\frac{\upsilon}{M_{1}+\alpha}\right),}~&t>0, x>0.
\end{array}\right.
\end{equation}
Assume that $u(t,x)=U(\xi)$ and $\upsilon(t,x)=V(\xi),$  where $~\xi=x-st $. Then \eqref{le2.6} is equivalent to
\begin{equation}\label{2.7}
\left\{\begin{array}{ll}
sU'+U''+U(1-U)=0,~\qquad&\xi \in \mathbb{R},\vspace{7 pt}\\
sV'+DV''+\kappa V\left(1-\frac{V}{M_{1}+\alpha}\right)=0,~~\qquad&\xi \in \mathbb{R},
\end{array}\right.
\end{equation}
If $s\geq s_{\rm min}=2\max \{1,\sqrt{D \kappa} \},$  then problem of \eqref{2.7} admits a solution $(U,V)$ which satisfies the conditions
\begin{equation}\label{limiting.cond}
\begin{array}{l}
U(-\infty)=1,~V(-\infty)=M_{1}+\alpha, \quad U(+\infty)=V(+\infty)=0,\vspace{6 pt}\\
U'(\xi)<0\text{ and }V'(\xi)< 0~\text{ for all }~\xi \in \mathbb{R}.
\end{array}
\end{equation}
\end{lemma}
The following lemma will be used to give a lower estimate of the ``asymptotic spreading speed'' (when spreading occurs). The notion of spreading and spreading speed will become more clear later on.

Before we state the needed lemma, let us  first consider the following problem (which is relevant to the original problem \eqref{FBP}. It will also initiate problem \eqref{2.77}, the subject of Lemma \ref{le2.7}.)
\begin{equation}\label{2.771}
\left\{\begin{array}{ll}
\ds{\partial_{t}\underline{\upsilon}-D\partial_{xx}\underline{\upsilon}= \kappa\underline{\upsilon}\left(1-\frac{\underline{\upsilon}}{\alpha}\right),}~~\qquad\qquad\qquad &t>0,0<x<\underline{h}(t),\vspace{7 pt}\\
\partial_{x}\underline{\upsilon}(t,0)= 0,~&t>0,\vspace{7 pt}\\
\underline{\upsilon}(t,\underline{h}(t))=0,~&t>0,\vspace{7 pt}\\
\underline{h}'(t)=-\mu\rho\,\partial_{x}\underline{\upsilon}(t,\underline{h}(t)),~&t>0.
\end{array}\right.
\end{equation}
We assume that $(\underline{\upsilon},\underline{h})$ is the unique solution of \eqref{2.771} and $\underline{h}(t)\rightarrow +\infty$ as $t\rightarrow +\infty$. Setting $$\omega(t,x)=\underline{\upsilon}(t, \underline{h}(t)-x),$$  we then obtain
\begin{equation}\label{2.772}
\left\{\begin{array}{ll}
\ds{\omega_{t}-D\omega_{xx}+\underline{h}'(t) \omega_{x}= \kappa\omega(1-\frac{\omega}{\alpha}),} &\text{for all }t>0\text{ and }0<x<\underline{h}(t),\vspace{7 pt}\\
\omega_{x}(t,\underline{h}(t))= 0,~&t>0,\vspace{7 pt}\\
\omega(t,0)=0,&t>0,\vspace{7 pt}\\
\underline{h}'(t)=\mu\rho\,\omega_{x}(t,0),~&t>0.
\end{array}\right.
\end{equation}

Since $\ds{\lim_{t\rightarrow +\infty}\underline{h}(t)=+\infty},$  if $\underline{h}'(t)$ approaches a constant $s_{*}$ and $\omega(t,x)$ approaches a positive function $V(x)$ as $t\rightarrow+\infty,$  then $V(x)$ must be a positive solution of \eqref{2.77} with $s_{*}=\mu\rho V'(0)$.

We now state the lemma.
\begin{lemma}[Proposition 4.1 in \cite{dushou}]\label{le2.7}
  For any $s\geq 0,$  the following problem
\begin{equation}\label{2.77}
\left\{\begin{array}{ll}
\ds{sV'-DV''-\kappa V\left(1-\frac{V}{\alpha}\right)=0,}& x>0 ,\\
V(0)=0,
\end{array}\right.
\end{equation}
admits a unique positive solution $V=V_{s}.$ Furthermore, for each $\mu,\rho > 0,$ there exists a unique $s_{*}$ such that $\mu\rho V'_{s_{*}}(0)=s_{*}$.
\end{lemma}

\section{The spreading-vanishing dichotomy}\label{dichotomy}
We have seen in Lemma \ref{estimates} that $h'(t)>0$ for all $t>0.$ This allows us to define \begin{equation}\label{define.hinfty}h_{\infty}:=\lim_{t\rightarrow+\infty}h(t)\text{ in }[0,+\infty)\cup \{\infty\}.\end{equation}
This will allow us to define the notions of spreading and vanishing as follows.
\begin{definition}\label{terminology}We say that the two species $u$ and $\upsilon$ {\em vanish eventually} if $h_{\infty}< \infty$ and
\[\lim_{t\rightarrow+\infty}\|u(t,\cdot)\|_{C([0,h(t)])}
=\lim_{t\rightarrow+\infty}\|\upsilon(t,\cdot)\|_{C([0,h(t)])}=0.\] We say that the two species $u$ and $\upsilon$ {\em spread successfully} if \[h_{\infty}=+\infty,~\ds{\liminf_{t\rightarrow+\infty}u(t,x)>0}\text{ and }\ds{\liminf_{t\rightarrow+\infty}\upsilon(t,x)>0}\] uniformly in any compact subset of $[0,+\infty)$.\end{definition}
\subsection{The Spreading Case}
The following theorem shows that $h_{\infty}=+\infty$ is sufficient for a successful spreading:
\begin{theorem}\label{th3.1}
Suppose that $(u,\upsilon,h(t))$ is the solution of \eqref{FBP}. If $h_{\infty}=+\infty,$ ~then we have
$$\lim_{t\rightarrow+\infty}u(t,x)=u^{*}~\text{ and }~\lim_{t\rightarrow+\infty}\upsilon(t,x)=\upsilon^{*}.$$
\end{theorem}
\begin{proof}We will divide the proof of this theorem into two steps.
\medskip

{\bf Step 1.} Since $h_{\infty}=+\infty,$  then for any $l_\varepsilon,$  there exists $T_1>0$ and $l_1>0$ such that $\ds{l_1>\max \left\{l_\varepsilon,\frac{\pi}{2}\right\}},$  when $t>T_1,$ and then $u$ satisfies
\begin{equation}\label{3.1}
\left\{\begin{array}{ll}
u_{t}-u_{xx}\leq u(1-u),&t>T_{1},~0<x<l_1,\vspace{7 pt}\\
u_{x}(t,0)=0,~u(t,l_1)\leq M,&t>T_{1},\vspace{7 pt}\\
u(T_{1},x)>0,&x\in [0,l_1),
\end{array}\right.
\end{equation}
where $M=\max\{M_{1},M_{2}\}$ (the constants appearing in \eqref{5.6}  and \eqref{5.7}.) Applying Lemma \ref{le2.4}, we  obtain that \[\ds{\limsup_{t\rightarrow +\infty}u(t,x)< 1+\varepsilon}\text{ uniformly in $[0,l_\varepsilon]$}.\]  Since $\varepsilon$ and $l_\varepsilon$ are arbitrary,  then $\ds{\limsup_{t\rightarrow +\infty}u(t,x)\leq 1=:\bar{u}_{1}}$ uniformly on $[0,+\infty)$.

Now let $\ds{l_2>\max\left\{l_\varepsilon,\frac{\pi}{2}\sqrt{\frac{D}{\kappa}}\right\}}.$ In view of the last conclusion, there exists $T_{2}>T_{1}$ such that $u(t,x)<\bar{u}_{1}+\varepsilon$ when $t> T_{2}$ and $0< x <l_2.$  Then $\upsilon$ satisfies
\begin{equation}\label{3.2}
\left\{\begin{array}{ll}
\ds{\upsilon_{t}-D\upsilon_{xx}\leq \kappa\upsilon\left(1-\frac{\upsilon}{\bar{u}_{1}+\varepsilon+\alpha}\right)},& t>T_{2},~0<x<l_2,\vspace{7 pt}\\
\upsilon_{x}(t,0)=0\text{ and }\upsilon(t,l_2)\leq M,& t>T_{2},\vspace{7 pt}\\
\upsilon(T_{2},x)>0, & x\in [0,l_2).
\end{array}\right.
\end{equation}
Applying Lemma \ref{le2.4} again, we get $\ds{\limsup_{t\rightarrow +\infty}\upsilon(t,x)<\bar{u}_{1} +\alpha+\varepsilon}$ uniformly on $[0,l_\varepsilon].$ The arbitrariness of $\varepsilon$ and $l_\varepsilon$  allows us to conclude that $\ds{\limsup_{t\rightarrow +\infty}\upsilon(t,x)\leq \bar{u}_{1}+\alpha=:\bar{\upsilon}_{1}},$ uniformly on $[0,+\infty)$.

Let $\ds{l_3>\max \left\{l_\varepsilon,\frac{\pi}{2}\right\}}.$ From the above conclusion, we know that there exists $T_{3}>T_{2}$ such that $\upsilon(t,x)<\bar{\upsilon}_{1}+\varepsilon$ and $u(t,x)>0$ whenver $t> T_{3}$ and $0< x <l_3.$  Then $u$ satisfies
\begin{equation}\label{3.3}
\left\{\begin{array}{ll}
u_{t}-u_{xx}\geq u(1-u)-\delta u(\bar{\upsilon}_{1}+\varepsilon),
& t>T_{3},~0<x<l_3,\vspace{7 pt}\\
u_{x}(t,0)=0,u(t,l_3)=0, & t>T_{3},\vspace{7 pt}\\
u(T_{3},x)>0,~&x\in [0,l_3).
\end{array}\right.
\end{equation}
By Lemma \ref{le2.3}, we get $\ds{\liminf_{t\rightarrow+\infty}u(t,x)>1-\delta\bar{\upsilon}_{1}-\varepsilon}$ uniformly on $[0,l_\varepsilon]$. Again using the arbitrariness of $\varepsilon$ and $l_\varepsilon,$  it follows that $\ds{\liminf_{t\rightarrow+\infty}u(t,x)\geq 1-\delta\bar{\upsilon}_{1}=:\underline{u}_{1}>0}$ because of the hypothesis {\bf (H1)}.

Let $\ds{l_4>\max\left\{l_\varepsilon,\frac{\pi}{2}\sqrt{\frac{D}{\kappa}}\right\}}$. In view of above result, then there exists $T_4>T_{3}$ such that $u(t,x)>\underline{u}_{1}-\varepsilon$  whenever $t> T_{4}$ and $0< x <l_4.$  Then $\upsilon$ satisfies
\begin{equation}\label{3.4}
\left\{\begin{array}{ll}
\ds{\upsilon_{t}-D\upsilon_{xx}\geq \kappa\upsilon\left(1-\frac{\upsilon}{\underline{u}_{1}-\varepsilon+\alpha}\right),}&t>T_{4},~0<x<l_4,\vspace{7 pt}\\
\upsilon_{x}(t,0)=0,\upsilon(t,l_4)=0, &t>T_{4},\vspace{7 pt}\\
\upsilon(T_{4},x)>0, &x\in [0,l_4).
\end{array}\right.
\end{equation}
Applying Lemma \ref{le2.3}, we have $\ds{\liminf_{t\rightarrow+\infty}\upsilon(t,x)>\underline{u}_{1}+\alpha-\varepsilon}$ uniformly on $[0,l_\varepsilon],$  and consequently (as  $\varepsilon$ and $l_\varepsilon$ are arbitrary)  we obtain $\ds{\liminf_{t\rightarrow+\infty}\upsilon(t,x)\geq\underline{u}_{1}+\alpha=:\underline{\upsilon}_{1}.}$

Now we will build a $\bar{u}_{2}$.

Denote $l_5>\max \left\{l_\varepsilon,\frac{\pi}{2}\right\}$.~By above conclusion, we know that there exists $T_{5}>T_{4}$~such that $\upsilon(t,x)>\underline{\upsilon}_{1}-\varepsilon$ ~when $t> T_{5},~0< x <l_5$, and then $u$ satisfies:
\begin{equation}\label{3.1aa}
\left\{\begin{array}{ll}
u_{t}-u_{xx}\leq u(1-u)-\delta u(\underline{\upsilon}_{1}-\varepsilon),~~~~~~~~
& t>T_{5},~0<x<l_5,\\
u_{x}(t,0)=0,u(t,l_5)=0, & t>T_{5},\\
u(T_{5},x)>0,~&x\in [0,l_5).
\end{array}\right.
\end{equation}
By Lemma \ref{le2.4}, we have $\limsup_{t\rightarrow+\infty}u(t,x)<1-\delta\bar{\upsilon}_{1}-\varepsilon$ uniformly on $[0,l_\varepsilon]$. Again using the arbitrariness of~$\varepsilon$~and~$l_\varepsilon$, it follows that $\liminf_{t\rightarrow+\infty}u(t,x)\leq 1-\delta\underline{\upsilon}_{1}=:\bar{u}_{2}>0$ uniformly on $[0,+\infty)$.

The construction of $\bar{\upsilon}_{2}$.

let $l_6>\max\left\{l_\varepsilon,\frac{\pi}{2}\sqrt{\frac{D}{\kappa}}\right\}$. In view of \eqref{3.1aa}, there exists $T_{6}>T_{5}$~such that~$u(t,x)<\bar{u}_{2}+\varepsilon$ when $t> T_{6},~0< x <l_6$, and then $\upsilon$ such that
\begin{equation}\label{3.2aa}
\left\{\begin{array}{ll}
\upsilon_{t}-D\upsilon_{xx}\leq \kappa\upsilon(1-\frac{\upsilon}{\bar{u}_{2}+\varepsilon+\alpha}),~~~~~~~& t>T_{6},~0<x<l_6,\\
\upsilon_{x}(t,0)=0,\upsilon(t,l_6)\leq M,& t>T_{6},\\
\upsilon(T_{6},x)>0, & x\in [0,l_6).
\end{array}\right.
\end{equation}
Applying Lemma \ref{le2.4}, we have~$\limsup_{t\rightarrow +\infty}\upsilon(t,x)<\bar{u}_{2} +\alpha+\varepsilon$~uniformly on~$[0,l_\varepsilon]$. Considering the arbitrariness of $\varepsilon$~and $l_\varepsilon$, we then have~$\limsup_{t\rightarrow +\infty}\upsilon(t,x)\leq \bar{u}_{2}+\alpha=:\bar{\upsilon}_{2}$, uniformly on~$[0,+\infty)$.

Furthermore, let $l_7>\max \left\{l_\varepsilon,\frac{\pi}{2}\right\}$.~By above conclusion, we know that there exists $T_{7}>T_{6}$~such that $\upsilon(t,x)<\bar{\upsilon}_{2}+\varepsilon$ and~$u(t,x)>0$~when $t> T_{7},~0< x <l_7$, and then $u$ satisfies:
\begin{equation}\label{3.3aa}
\left\{\begin{array}{ll}
u_{t}-u_{xx}\geq u(1-u)-\delta u(\bar{\upsilon}_{2}+\varepsilon),~~~~~~~~
& t>T_{7},~0<x<l_7,\\
u_{x}(t,0)=0,u(t,l_7)=0, & t>T_{7},\\
u(T_{7},x)>0,~&x\in [0,l_7).
\end{array}\right.
\end{equation}
By Lemma \ref{le2.3}, we have $\liminf_{t\rightarrow+\infty}u(t,x)>1-\delta\bar{\upsilon}_{2}-\varepsilon$ uniformly on $[0,l_\varepsilon]$. Again using the arbitrariness of~$\varepsilon$~and~$l_\varepsilon$, it follows that $\liminf_{t\rightarrow+\infty}u(t,x)\geq 1-\delta\bar{\upsilon}_{2}=:\underline{u}_{2}$.

In order to sharpen the upper and lower bounds above, we continue to use the above approach and find
$l_8>\max\left\{l_\varepsilon,\frac{\pi}{2}\sqrt{\frac{D}{\kappa}}\right\}$.~In view of above result, then there exists $T_8>T_{7}$~such that $u(t,x)>\underline{u}_{2}-\varepsilon$, when $t> T_{8},~0< x <l_8$, and then $\upsilon$ satisfies
\begin{equation}\label{3.4aa}
\left\{\begin{array}{ll}
\upsilon_{t}-D\upsilon_{xx}\geq \kappa\upsilon(1-\frac{\upsilon}{\underline{u}_{2}-\varepsilon+\alpha}),~~~~~~~~~~&t>T_{8},~0<x<l_8,\\
\upsilon_{x}(t,0)=0,\upsilon(t,l_8)=0, &t>T_{8},\\
\upsilon(T_{8},x)>0, &x\in [0,l_8).
\end{array}\right.
\end{equation}
Applying Lemma \ref{le2.3}, we have~$\liminf_{t\rightarrow+\infty}\upsilon(t,x)>\underline{u}_{1}+\alpha-\varepsilon$~uniformly on~$[0,l_\varepsilon]$, because of the arbitrariness of~$\varepsilon$~and~$l_\varepsilon$, it implies that~$\liminf_{t\rightarrow+\infty}\upsilon(t,x)\geq\underline{u}_{2}+\alpha=:\underline{\upsilon}_{2}$.

\medskip

{\bf Step 2.} Indeed, we can continue the above strategy to obtain the following sequences, whose monotonicity is a straightforward conclusion  $$\ds{\underline{u}_{1}\leq\ldots\leq\underline{u}_{i}\leq\ldots\leq \liminf_{t\rightarrow+\infty}u(t,x)
\leq\limsup_{t\rightarrow+\infty}u(t,x)\leq\ldots\leq\bar{u}_{i}
\leq\ldots\leq\bar{u}_{1},}$$
$$\ds{\underline{\upsilon}_{1}\leq\ldots\leq\underline{\upsilon}_{i}
\leq\ldots\leq\liminf_{t\rightarrow+\infty}\upsilon(t,x)
\leq\limsup_{t\rightarrow+\infty}\upsilon(t,x)
\leq\ldots\leq\bar{\upsilon}_{i}\leq\ldots\leq\bar{\upsilon}_{1},}$$
where $\ds{\underline{u}_{i}=1-\delta\bar{\upsilon}_{i}},$  $\ds{\bar{u}_{i}=1-\delta\underline{\upsilon}_{i-1},}$
 $\ds{\underline{\upsilon}_{i}=\underline{u}_{i}+\alpha}$ and  $\ds{\bar{\upsilon}_{i}=\bar{u}_{i}+\alpha}$ for $i=1,2,3,\cdots$.
 \medskip

Since the constant sequences $\{\bar{u}_{i}\}$ and $\{\bar{\upsilon}_{i}\}$ are monotone non-increasing and bounded from below, and the sequences $\{\underline{u}_{i}\}$ and $\{\underline{\upsilon}_{i}\}$ are monotone non-decreasing, and are bounded from above, the limits of these sequences exist. Let us denote their limits, as $i\rightarrow+\infty,$ by $\bar{u},$  $\bar{\upsilon},$  $\underline{u}$ and $\underline{\upsilon}$  respectively. We then have \[\bar{u}=1-\delta\underline{\upsilon},~\underline{u}=1-\delta\bar{\upsilon},~\bar{\upsilon}=\bar{u}+\alpha \text{ and }\underline{\upsilon}=\underline{u}+\alpha.\] Thus,
\begin{equation}\label{3.5}
\left\{\begin{array}{l}
\bar{u}=1-\delta(\underline{u}+\alpha),\vspace{7 pt}\\
\underline{u}=1-\delta(\bar{u}+\alpha).
\end{array}\right.
\end{equation}
From hypothesis {\bf (H1)}, we can easily conclude that $\bar{u}=\underline{u}=u^{*}$ and this implies that \[\ds{\liminf_{t\rightarrow+\infty}u(t,x)=\limsup_{t\rightarrow+\infty}u(t,x)=u^{*}}\text{ and  }\ds{\liminf_{t\rightarrow+\infty}\upsilon(t,x)=\limsup_{t\rightarrow+\infty}
  \upsilon(t,x)=\upsilon^{*}}.\]The proof of  Theorem \ref{th3.1} is now complete. \end{proof}

\subsection{The Vanishing Case}
The following theorem shows that the finiteness of $h_{\infty}$ leads both species, $u$ and $\upsilon$, to vanish.
\begin{theorem}\label{th3.2}
Let $(u,\upsilon,h(t))$ be the solution of \eqref{FBP}. If $h_{\infty}<\infty,$  then we have
$\ds{\lim_{t\rightarrow+\infty}\| u(t,\cdot)\|_{C[0,h(t)]}=0}$ and $\ds{\lim_{t\rightarrow+\infty}\| \upsilon(t,\cdot)\|_{C[0,h(t)]}=0.}$
\end{theorem}
\begin{proof} Since $u(t,x)>0$ and $u_{x}(t,h(t))<0,$ ~then $\upsilon$ satisfies
\begin{equation}\label{3.6}
\left\{\begin{array}{ll}
\ds{\upsilon_{t}-D\upsilon_{xx}\geq \kappa\upsilon\left(1-\frac{\upsilon}{\alpha}\right),}&\text{for all }t>0\text{ and }0<x<h(t),\vspace{7 pt}\\
\upsilon_{x}(t,0)=0, &t>0,\vspace{7 pt}\\
\upsilon(t,h(t))=0, h'(t)\geq-\mu\rho \upsilon_{x}(t,h(t)),&t>0 \vspace{7 pt}\\
\upsilon(0,x)=\upsilon_{0}(x),  & x\in[0,h_{0}].
\end{array}\right.
\end{equation}
In view of Lemmas \ref{le2.5} and \ref{le2.8},  we have that $\ds{\lim_{t\rightarrow+\infty}
\|\upsilon(t,\cdot)\|_{C[0,h(t)]}=0}$. Hence, there exists $T>0$ such that $\upsilon(t,x)< \varepsilon$ for all $t\geq T$ and  $ 0\leq x \leq h(t),$
where $0<\varepsilon<<1$. Since $u(t,x)>0$ and $\upsilon_{x}(t,h(t))<0,$  then
\begin{equation}\label{3.7}
\left\{\begin{array}{ll}
u_{t}-u_{xx}\geq u(1-\delta\varepsilon-u),& t>T, 0<x<h(t),\vspace{7 pt}\\
u_{x}(t,0)=0, & t>T,\vspace{7 pt}\\
u(t,h(t))=0, h'(t)\geq-\mu u_{x}(t,h(t)),& t>T,\vspace{7 pt}\\
u(T,x)=u_{0}(x), & x\in[0,h_{0}].
\end{array}\right.
\end{equation}
Applying Lemmas \ref{le2.5} and \ref{le2.8}, we obtain that $\ds{\lim_{t\rightarrow+\infty}\| u(t,\cdot)\|_{C[0,h(t)]}=0}$.\end{proof}
\subsection{Sharp criteria for spreading and vanishing}
In this section, we derive some criteria governing the spreading and vanishing for the free-boundary problem \eqref{FBP}.
\begin{lemma}\label{le3.1}
If $h_{\infty}<\infty,$  then $\ds{h_{\infty}\leq\frac{\pi}{2}\min\left\{1, \sqrt{\frac{D}{\kappa}}\right\}:=h_{*}.}$ Furthermore, $h_{0}\geq h_{*}$ implies that $h_{\infty}=+\infty$.
\end{lemma}
\begin{proof} The proof of Lemma \ref{le3.1} is essentially the same as that of Theorem 5.1 in \cite{w3}. By Theorem \ref{th3.2}, we know that if $h_{\infty}<\infty,$ then
$$\lim_{t\rightarrow+\infty}\| u(t,\cdot)\|_{C[0,h(t)]}=0,\lim_{t\rightarrow+\infty}\| \upsilon(t,\cdot)\|_{C[0,h(t)]}=0.$$
In the following, we assume that $\ds{h_{\infty}>\frac{\pi}{2}\min\left\{1, \sqrt{\frac{D}{\kappa}}\right\}}$ to get the contradiction.

First, as $\ds{h_{\infty}>\frac{\pi}{2}},$  there exists $\varepsilon>0$ such that $\ds{h_{\infty}>\frac{\pi}{2}\sqrt{\frac{1}{1-\delta\varepsilon}}}.$ For such $\varepsilon,$ there exists $T>0$ such that $\ds{h(T)>\frac{\pi}{2}\sqrt{\frac{1}{1-\delta\varepsilon}}}$ and
  $\upsilon(t,x)\leq\varepsilon,$ for $t>T$ and $x\in[0,h(T)].$ Let $\underline{u}(t,x)$ be the solution of the following problem:
\begin{equation}\label{3.8}
\left\{\begin{array}{ll}
\partial_{t}\underline{u}-\partial_{xx}\underline{u}=\underline{u}(1-\delta\varepsilon-
\underline{u}),&\text{for }t>T\text{ and }0<x<h(T),\vspace{7 pt}\\
\partial_x\underline{u}(t,0)=\underline{u}(t,h(T))=0,&t>T,\vspace{7 pt}\\
\underline{u}(T,x)=u(T,x), &0<x<h(T).
\end{array}\right.
\end{equation}
By the comparison principle, we have $\underline{u}(t, x)\leq u(t, x),$  for all $t>T$ and $0<x<h(T)$. Since $\ds{h(T)>\frac{\pi}{2}\sqrt{\frac{1}{1-\delta\varepsilon}}},$  the Proposition 3.2 of \cite{Cantrell2003} yields $\ds{\liminf_{t\rightarrow+\infty}u(t,x)\geq\liminf_{t\rightarrow+\infty}\underline{u}(t,x)>0,}$ which is a contradiction to Theorem \ref{th3.2}.

Secondly, as $\ds{h_{\infty}>\frac{\pi}{2}\sqrt{\frac{D}{\kappa}}},$ there exists $T>0$ such that $\ds{h(T)>\frac{\pi}{2}\sqrt{\frac{D}{\kappa}}}$ and $u(t,x)>0,$
 for all $t>T$ and $0<x<h(T).$ Let $\underline{\upsilon}(t, x)$ be the solution of the following equation
\begin{equation}\label{3.10}
\left\{\begin{array}{ll}
\ds{\partial_t\underline{\upsilon}-D\partial_{xx}\underline{\upsilon}= \kappa\underline{\upsilon}\left(1-\frac{\underline{\upsilon}}{\alpha}\right),}&t>T,~0<x<h(T),\vspace{7 pt}\\
\partial_x\underline{\upsilon}(t,0)=\underline{\upsilon}(t, h(T))=0, &t>T,\vspace{7 pt} \\
\underline{\upsilon}(T, x)=\upsilon(T, x), ~&0<x<h(T).
\end{array}\right.
\end{equation}
By the comparison principle, we have $\underline{\upsilon}(t, x)\leq\upsilon(t, x),$  for all $t>T$ and $0<x<h(T)$. Since $\ds{h(T)>\frac{\pi}{2}\sqrt{\frac{D}{\kappa}}},$ by the Proposition 3.2 of \cite{Cantrell2003} , we have \[\ds{\liminf_{t\rightarrow+\infty}\upsilon(t,x)\geq\liminf_{t\rightarrow+\infty}\underline{\upsilon}(t,x)>0,}\] which is a contradiction to Theorem \ref{th3.2}.
\medskip

Finally, since $h'(t)>0$ for all $t>0,$  then together with the above arguments we can  see that $h_{\infty}=+\infty$ when $\ds{h_{0}\geq\frac{\pi}{2}\min\left\{1, \sqrt{\frac{D}{\kappa}}\right\}}$. \end{proof}
\begin{lemma}\label{le3.2}
Suppose that the initial datum $h_{0}$ in problem \eqref{FBP} is such that $h_{0}<h_{*}.$ Then, there exists $\bar{\mu}>0$ depending on $u_{0}$ and $\upsilon_{0}$ such that $h_{\infty}=+\infty$ when $\mu\geq\bar{\mu}.$  More precisely, we have
 $$\bar{\mu}=\mu_1:=\frac{D}{\rho}\max\left\{1,\frac{\|\upsilon_{0}\|_{\infty}}{\alpha} \right\}\left(\frac{\pi}{2}\sqrt{\frac{D}{\kappa}}
 -h_{0}\right)\left(\int_{0}^{h_{0}}\upsilon_{0}(x) dx\right)^{-1}.$$
 Furthermore, if $\|\upsilon_{0}\|_{\infty}\leq 1+\theta$ and $\|u_0\|_{\infty}\leq 1,$  then~
 $\bar{\mu}=\min \left\{\mu_1 , \mu_2\right\},$  where \[\ds{\mu_2=\max\left\{1, \frac{\| u_{0} \|_{\infty}}{1-\delta(1+\theta)}\right\}\left(\frac{\pi}{2}-h_{0}\right)
 \left(\int_{0}^{h_{0}}u_{0}(x) dx \right)^{-1}}.\]
\end{lemma}
\begin{proof} We consider the following problem:
 \begin{equation}\label{3.12}
\left\{\begin{array}{ll}
\ds{\partial_t\underline{\upsilon}-D\partial_{xx}\underline{\upsilon}=
\kappa\underline{\upsilon}\left(1-\frac{\underline{\upsilon}}{\alpha}\right),}
&t>0,0<x<\underline{h}(t),\vspace{7 pt}\\
\partial_x\underline{\upsilon}(t, 0)=0,&t>0,\vspace{7 pt}\\
\underline{\upsilon}(t,\underline{h}(t))=0,&t>0,\vspace{7 pt}\\
\underline{h}'(t)=-\mu\rho\underline{\upsilon}(t,\underline{h}(t)),&t>0,\vspace{7 pt}\\
\underline{\upsilon}(0,x)=\upsilon_{0}(x),&0\leq x \leq h_{0},\vspace{7 pt}\\
h_{0}=\underline{h}(0), &t=0.
\end{array}\right.
\end{equation}
By Lemma \ref{le2.2}, we have $\underline{h}(t)\leq h(t)$ and  $ \underline{\upsilon}(t, x)\leq\upsilon(t,x),$ for $t>0$ and  $0< x <\underline{h}(t)$.
Using Lemma 3.7 of \cite{dushou}, if $\ds{\underline{h}(0)=h_{0}<h_{*}\leq \frac{\pi}{2}\sqrt{\frac{D}{\kappa}}}$ and $\mu\geq\bar{\mu},$ we have $\underline{h}(\infty)=+\infty.$ It then follows that $h_{\infty}=+\infty$.

Suppose now that $\|\upsilon_{0}\|_{\infty}\leq 1+\theta$ and $\|u_{0}\|_{\infty}\leq 1.$  That is $M_2=1+\theta.$  We consider the following problem
\begin{equation}\label{3.13}
\left\{\begin{array}{ll}
\ds{\partial_t\underline{u}-\partial_{xx}\underline{u}=\underline{u}(1-\delta(1+\theta)-\underline{u}),}
&t>0,0<x<\underline{h}(t),\vspace{7 pt}\\
\partial_x\underline{u}(t,0)=0,&t>0,\vspace{7 pt}\\
\underline{u}(t,\underline{h}(t))=0,&t>0,\vspace{7 pt}\\
\underline{h}'(t)=-\mu\underline{u}(t,\underline{h}(t)),&t>0,\vspace{7 pt}\\
\underline{u}(0,x)=u_{0}(x), &0\leq x \leq \underline{h}(0),\vspace{7 pt}\\
\underline{h}(0)=h_{0}, &t=0.
\end{array}\right.
\end{equation}
From Lemma 3.7 of \cite{dushou}, we know that $\ds{\underline{h}(0)=h_{0}<h_{*}\leq \frac{\pi}{2}}$ and $\mu\geq \mu_2$, which imply that $\underline{h}(\infty)=+\infty$. Thus $\mu\geq \min \{\mu_1 , \mu_2\}$ implies that $\underline{h}(\infty)=+\infty$. Therefore, we have $h_{\infty}=+\infty$ when $\mu\geq \min \{\mu_1 , \mu_2\}$. \end{proof}

\begin{lemma}\label{le3.3}
 Suppose that the initial datum $h_{0},$ in problem \eqref{FBP}, is such that $h_{0}<h_{*}.$ Then, there exists $\underline{\mu}>0$ depending on $u_{0}(x)$ and $\upsilon_{0}(x)$ such that $h_{\infty}<\infty$ when $\mu\leq \underline{\mu}$.
\end{lemma}
\begin{proof}  We adopt the same method used to prove Lemma 5.2 of \cite{w3}, Lemma 3.8 of \cite{dushou} and Corollary 1 of \cite{b}. Let  $\ds{\varepsilon=\frac{1}{2}\left(\frac{h_{*}}{h_{0}}-1\right)>0}$  since $h_{0}<h_{*}$. Define
$$\bar{h}(t)=h_{0}(1+\varepsilon-\frac{\varepsilon}{2}e^{-\beta t})~~\text{ for }~t\geq 0$$
$$V(y)=\ds{\cos\frac{\pi y}{2}}~~ \text{ for }~0\leq y \leq 1;$$
and
$$\bar{u}(t,x)=\bar{\upsilon}(t,x)=\widetilde{M}e^{-\beta t}V\left(\frac{x}{\bar{h}(t)}\right) ~~\text{ for }~~0\leq x \leq \bar{h}(t),$$
where
$\ds{\beta=\frac{1}{2} \min\left\{\left(\frac{\pi}{2}\right)^{2}\frac{D}{(1+\varepsilon)^{2}h_{0}^{2}}
-\kappa,\left(\frac{\pi}{2}\right)^{2}\frac{1}{(1+\varepsilon)^{2}h_{0}^{2}}-1 \right\}>0},$ as $h_{0}(1+\varepsilon)<h_{*}$ and $$\ds{\widetilde{M}=\frac{\max\left\{\|u_{0}\|_{\infty},\| \upsilon_{0}\|_{\infty}\right\}}{\ds{\cos\left(\frac{\pi}{2+\varepsilon}\right)}}.}$$

If $\ds{\mu\leq \underline{\mu}=\frac{\varepsilon h_{0}^{2}\beta (2+\varepsilon)}{2 (1+\rho)\pi \widetilde{M}}}$\,, then a direct computation yields
\begin{equation}\label{3.14}
\left\{\begin{array}{ll}
\bar{u}_{t}-\bar{u}_{xx}-\bar{u}(1-\bar{u})\geq \widetilde{M}e^{-\beta t}V((\frac{\pi}{2})^{2}\frac{1}{(1+\varepsilon)^{2}h_{0}^{2}}-1-\beta)\geq0,~&t>0,~0<x<\bar{h}(t),\vspace{7 pt}\\
\ds{\bar{\upsilon}_{t}-D\bar{\upsilon}_{xx}-\kappa\bar{\upsilon}\left(1-\frac{\bar{\upsilon}}{M_{1}+\alpha}\right)\geq \widetilde{M}e^{-\beta t}V((\frac{\pi}{2})^{2}\frac{D}{(1+\varepsilon)^{2}h_{0}^{2}}-\kappa-\beta)\geq 0,}&t>0,~0<x<\bar{h}(t),\vspace{7 pt}\\
\bar{u}_{x}(t,0)=\bar{\upsilon}_{x}(t,0)=0,~&t>0,\vspace{7 pt}\\
\bar{u}(t,\bar{h}(t))=\bar{\upsilon}(t,\bar{h}(t))=0,~&t>0,\vspace{7 pt}\\
\ds{\bar{h}'(t)+\mu[\bar{u}_{x}(t,\bar{h}(t))+\rho\bar{\upsilon}_{x}(t,\bar{h}(t))]\geq \frac{\varepsilon h_{0} \beta e^{-\beta t}}{2}\left(1-\frac{2\mu (1+\rho)\pi \widetilde{M}}{\varepsilon h_{0}^{2}\beta (2+\varepsilon)}\right)\geq 0,}~&t>0.
\end{array}\right.
\end{equation}
  Since $h_{0}\leq \bar{h}(0),$ $\bar{u}(0,x)\geq u_{0}(x)$ and $\bar{\upsilon}(0,x)\geq \upsilon_{0}(x)$ for all $x\in[0,h_{0}],$ then  Lemma \ref{le2.2}  yields that $h(t)\leq \bar{h}(t)$ \text{ on } $[0,+\infty)$. Taking $t\rightarrow +\infty,$ we obtain  $$h_\infty\leq \bar{h}(\infty)=h_{0}(1+\delta)<h_{*}.$$ This, together with Lemma \ref{le3.1}, complete the proof. \end{proof}

 Lemmas \ref{le3.1} and \ref{le3.3} lead to {\bf other criteria} for spreading and vanishing, in terms of the parameter $D,$ when  $h_{0}$ is fixed.
\begin{lemma}\label{le3.4}
For a fixed $h_{0}>0,$ let $D^{*}=\ds{\frac{4\kappa h^{2}_{0}}{\pi^{2}}}.$  Then,
\begin{enumerate}[\rm (i)]
 \item if $0 < D\leq D^{*} ,$  spreading occurs (see Definition \ref{terminology}). \\
 \item Suppose that $D^{*}< D \leq \kappa$. If $\mu\geq\bar{\mu},$ then the spreading occurs. If $\mu\leq \underline{\mu},$  then  vanishing occurs (see Definition \ref{terminology}).
 \end{enumerate}
\end{lemma}

\section{Spreading speed}\label{spreading.estimate}
In this section, we derive upper and lower bounds for the spreading speed under the free boundary conditions stated in \eqref{FBP}. The estimates are given in terms of well-known parameters.
\begin{theorem}\label{spreading.speed}
Let $(u,\upsilon,h)$ be the solution of problem \eqref{FBP} with $h_{\infty}=\infty$ and recall that  $$s_{\rm min}=2 \max \left\{1,\sqrt{D\kappa}\right\}.$$ Then,
$$\ds{s_{*}\leq\liminf_{t\rightarrow +\infty}\frac{h(t)}{t}\leq\limsup_{t\rightarrow +\infty}\frac{h(t)}{t}\leq s_{\rm min},}$$
where $s_{*}$ is the constant appearing  in Lemma \ref{le2.7}.
\end{theorem}

  \begin{proof}[Proof of Theorem \ref{spreading.speed}] First we will prove $\ds{\limsup_{t\rightarrow +\infty}\frac{h(t)}{t}\leq s_{\rm min}}.$
 From Lemma \ref{le2.6}, we know that $(U(\xi), V(\xi))\rightarrow (0, 0)$ and $(U'(\xi), V'(\xi))\rightarrow (0, 0)$ as $\xi\rightarrow+\infty$. Then,   we can choose $l$ and  $g\gg 1$ such that
 \begin{equation}\label{4.01}
 l U(\xi)\geq \| u_{0}\|_{\infty},~~g V(\xi)\geq \| \upsilon_{0}\|_{\infty} \text{ for all }\xi\in[0,h_{0}].
 \end{equation}
Moreover, there exists $\sigma_{0} > h_{0}$ depending on $D, \kappa, \mu, \rho$ such that
\begin{equation}\label{4.02}
U(\sigma_{0})< \min_{0\leq x\leq h_{0}}\left(U(x)-\frac{u_{0}(x)}{l}\right),~~~~V(\sigma_{0})< \min_{0\leq x\leq h_{0}}\left(V(x)-\frac{\upsilon_{0}(x)}{g}\right),
 \end{equation}
 \begin{equation}\label{4.03}
U(\sigma_{0})\leq 1-\frac{1}{l},~~~V(\sigma_{0})\leq\left (1-\frac{1}{g}\right)(M_{1}+\alpha),
 \end{equation}
 and
 \begin{equation}\label{4.04}
-\mu(lU'(\sigma_{0})+g\rho V'(\sigma_{0}))< s_{\rm min}.
 \end{equation}

 \noindent Now let $\sigma(t)=\sigma_{0}+s_{\rm min}t~$ for $~t\geq 0,$ \[\bar{u}=lU(x-s_{\rm min}t)-lU(\sigma_{0})~\text{ and }~ \bar{\upsilon}=gV(x-s_{\rm min}t)-gV(\sigma_{0})\text{ for }t\geq 0 \text{ and }  0\leq x\leq \sigma(t).\]  It is obvious from \eqref{4.02} and \eqref{4.04} that
  $$\bar{u}(0,x)>u_{0}(x),~\bar{\upsilon}(0,x)>\upsilon_{0}(x),~~\text{for}~~0\leq x\leq h_{0};$$
  and
  $$\sigma'(t)=s_{\rm min}> -\mu(\bar{u}_{x}(t,\sigma(t))+\rho\bar{\upsilon}_{x}(t,\sigma(t))).$$
Moreover, $$\bar{u}(t,\sigma(t))=\bar{\upsilon}(t,\sigma(t))=0~~\text{for all}~t\geq 0;$$
  $$\bar{u}_{x}(t,0)<0,~\bar{\upsilon}_{x}(t,0)< 0~~\text{for all}~t\geq 0 ~\text{ (by~Lemma~\ref{le2.6}).}$$
Then by a  calculation, we obtain from \eqref{4.03} that
\begin{align*}
  \bar{u}_{t}-\bar{u}_{xx}-\bar{u}(1-\bar{u}) &=l\left[(l-1)\left(U-\frac{lU(\sigma_{0})}{l-1}\right)^{2}+U(\sigma_{0})\frac{l-1-lU(
\sigma_{0})}{l-1}\right] \\
  &\geq 0 ,
\end{align*}
and
\begin{align*}
&\bar{\upsilon}_{t}-D\bar{\upsilon}_{xx}-\kappa\bar{\upsilon}
\left(1-\frac{\bar{\upsilon}}{M_{1}+\alpha}\right)\\
&=\frac{g\kappa}{M_{1}+\alpha}\left[(g-1)\left(V-\frac{gV(\sigma_{0})}{g-1}\right)^{2}
+V(\sigma_{0})\frac{(g-1)(M_{1}+\alpha)-gV(\sigma_{0})}{g-1}\right]\\
&\geq 0.
\end{align*}
 Then, by Lemma \ref{le2.2}, we have $h(t)\leq\sigma(t)~ \text{for}~t\geq 0$. Therefore, $$\ds{\limsup_{t\rightarrow +\infty}\frac{h(t)}{t} \leq \lim_{t\rightarrow +\infty}\frac{\sigma(t)}{t}= s_{\rm min}.}$$

Now, we prove $\ds{\liminf_{t\rightarrow +\infty}\frac{h(t)}{t} \geq s_{*}}.$ Let $(\underline{\upsilon},\underline{h})$ be the solution of the free boundary problem
\begin{equation}\label{4.1}
\left\{\begin{array}{ll}
\ds{\partial_{t}\underline{\upsilon}-D\partial_{xx}\underline{\upsilon}= \kappa\underline{\upsilon}\left(1-\frac{\underline{\upsilon}}{\alpha}\right),}&t>0,
0<x<\underline{h}(t),\vspace{7 pt}\\
\underline{\upsilon}_{x}(t,0)= 0,&t>0,\vspace{7 pt}\\
\underline{\upsilon}(t,\underline{h}(t))=0,&t>0,\vspace{7 pt}\\
\underline{h}'(t)=-\mu\rho\,\partial_{x}\underline{\upsilon}(t,\underline{h}(t)),&t>0.
\end{array}\right.
\end{equation}
  By the comparison principle, we then have $\underline{h}(t)\leq h(t).$ From Theorem 4.2 in \cite{dushou}, we have $$\ds{s_{*}=\lim_{t\rightarrow +\infty}\frac{\underline{h}(t)}{t}\leq\liminf_{t\rightarrow +\infty}\frac{h(t)}{t}.}$$
  \end{proof}

 \section{Proof of existence and uniqueness}\label{exist.}
 This section is devoted to prove the results about local existence and uniqueness of the solution to the main problem \eqref{FBP}.
 \begin{proof}[Proof of Lemma \ref{local.exist}]
 The main idea is adapted from \cite{Chenxf}.  Let $\zeta \in C^3([0,\infty))$ such that
$\zeta(y)=1$ if $| y-h_{0}|\leq \frac{h_{0}}{4},$  $\zeta(y)=0$ if $| y-h_{0}|> \frac{h_{0}}{2},$  $|\zeta'(y)|\leq \frac{6}{h_{0}},$  for all $y$. Define
 \begin{equation}\label{5.1}
x=y+\zeta(y)(h(t)-h_{0}),~~~0\leq y<+\infty.
\end{equation}
Note that, as long as $\ds{| h(t)-h_{0}|\leq \frac{h_{0}}{8}},$ $(x,t)\longrightarrow(y,t)$ is a diffeomorphism from $[0,+\infty)$ to $[0,+\infty)$.
Moreover,
\begin{equation}\label{5.2}
0\leq x\leq h(t)\Leftrightarrow 0\leq y \leq h_{0}\text{ and }
x=h(t) \Leftrightarrow y=h_{0}.
\end{equation}

We then compute
$$\frac{\partial y}{\partial x}=\frac{1}{1+\zeta'(y)(h(t)-h_{0})}=A(h(t),y(t)), $$
$$\frac{\partial^{2}y}{\partial x^{2}}=\frac{-\zeta''(y)(h(t)-h_{0})}
{[1+\zeta'(y)(h(t)-h_{0})]^{3}}=B(h(t),y(t)),$$
$$\frac{\partial y}{\partial t}=\frac{-h'(t)\zeta(y)}{1+\zeta'(y)(h(t)-h_{0})}=C(h(t),y(t)).$$
Now, we denote $$U(t,y(t))=u(t,x),  ~~V(t,y(t))=\upsilon(t,x),~~F(U,V)= U(1-U-\delta V)~\text{ and }~G(U,V)=\kappa V\left(1-\frac{V}{U+\alpha}\right).$$
Then problem \eqref{FBP} becomes
\begin{equation}\label{5.3}
\left\{\begin{array}{ll}
\ds{\frac{\partial U}{\partial t}=A^{2}U_{yy}+(B-C) U_{y}+F(U,V),} \qquad\qquad&t>0, 0<y<h_{0},\vspace{7 pt}\\
\ds{\frac{\partial V}{\partial t}= DA^{2}V_{yy}+(DB-C)V_{y}+G(U,V),} &t>0,0<y< h_{0},\vspace{7 pt}\\
\ds{U_{y}(t,0)=V_{y}(t,0)=U(t,h_{0})=V(t,h_{0})=0,} &t>0,\vspace{7 pt}\\
\ds{h'(t)=-\mu (U_{y}(t,h_{0})+\rho V_{y}(t,h_{0})),}&t>0,\vspace{7 pt}\\
\ds{U(0,y)=U_{0}(y)=u_{0}(y),} &y\in[0,h_{0}],t=0,\vspace{7 pt}\\
\ds{V(0,y)=V_{0}(y)=\upsilon_{0}(y),} &y\in[0,h_{0}],t=0.
\end{array}\right.
\end{equation}

We denote by $\tilde{h}=-\mu (U'_{0}(h_{0})+\rho V'_{0}(h_{0})).$  As in \cite{linzhigui2007}, we shall prove the local existence by using the contraction mapping theorem. We let $T$ such that $0<T\leq \frac{h_{0}}{8(1+\tilde{h})}$ and introduce the function spaces
\begin{equation}
\begin{array}{l}
X_{1T}:=\{ U\in C(\mathcal{R}):U(0,y)=U_{0}(y), \| U-U_{0}\|_{C(\mathcal{R})} \leq 1 \},\vspace{7 pt}\\\nonumber
X_{2T}:=\{ V\in C(\mathcal{R}):V(0,y)=V_{0}(y), \| V-V_{0}\|_{C(\mathcal{R})} \leq 1 \},\vspace{7pt}\\\nonumber
X_{3T}:=\{ h\in C^{1}[0,T],\| h'-\tilde{h}\|_{C[0,T]}\leq 1 \},\nonumber
\end{array}
\end{equation}
where \[\mathcal{R}=\{(t,y):  0 \leq  t \leq  T, 0 < y < h_{0}\}.\] Then, the space $X_{T}=X_{1T}\times X_{2T}\times X_{3T}$ is a complete metric space, with the metric
\[d((U_{1},V_{1},h_{1}),(U_{2},V_{2},h_{2}))=\| U_{1}-U_{2}\|_{C(\mathcal{R})}+\| V_{1}-V_{2}\|_{C(\mathcal{R})}+\| h'_{1}-h'_{2}\|_{C[0,T]}.\]  We have
    $$| h(t)-h_{0}|\leq \int^T_{0}|h'(s)|ds \leq T(1+\tilde{h})\leq \frac{h_{0}}{8},$$
 so that the mapping $(t,x) \rightarrow (t,y)$ is diffeomorphism.

 As mentioned above, we will construct a contraction mapping from $X_{T}$ into $X_{T}$ in order to prove the existence of a local solution. We begin this construction now. As $0\leq t\leq T,$ the coefficients $A$, $B$ and $C$ are bounded and $A^{2}$ is between two positive constants. By standard $L^{p}$ theory and the Sobolev imbedding theorem, for any $(U,V,h)\in X_{T},$  the following initial boundary value problem
\begin{equation}\label{5.4}
\left\{\begin{array}{ll}
\ds{\frac{\partial \hat{U}}{\partial t}=A^{2}\hat{U}_{yy}+(B-C) \hat{U}_{y}+F(U,V),} &t>0, 0<y<h_{0},\vspace{7 pt}\\
\ds{\frac{\partial \hat{V}}{\partial t}= DA^{2}\hat{V}_{yy}+(DB-C)\hat{V}_{y}+G(U,V),} \qquad\qquad&t>0,0<y<h_{0},\vspace{7 pt} \\
\ds{\hat{U}_{y}(t,0)=\hat{V}_{y}(t,0)=0,} &t>0,\vspace{7 pt}\\
\ds{\hat{U}(t,h_{0})=\hat{V}(t,h_{0})=0,} &t>0,\vspace{7 pt}\\
\ds{\hat{U}(0,y)=U_{0}(y)=u_{0}(y),} &y\in[0,h_{0}],\vspace{7 pt}\\
\ds{\hat{V}(0,y)=V_{0}(y)=\upsilon_{0}(y),}&y\in[0,h_{0}],
\end{array}\right.
\end{equation}
for any $\theta\in (0,1),$ admits a unique bounded solution $(\hat{U},\hat{V})\in C^{\frac{(1+\theta)}{2},1+\theta}(\mathcal{R})\times C^{\frac{(1+\theta)}{2},1+\theta}(\mathcal{R}).$ Moreover, \[\| \hat{U}\|_{C^{\frac{(1+\theta)}{2},1+\theta}(\mathcal{R})}\leq C_{1}\text{ and } \|\hat{V}\|_{C^{\frac{(1+\theta)}{2},1+\theta}(\mathcal{R})}\leq C_{2},\]
where the constants $C_{1}$ and $C_{2}$ depend on $h_{0},$  $\theta,$  $\| U_{0}\|_{C^{2}[0,h_{0}]}$ and  $\| V_{0}\|_{C^{2}[0,h_{0}]}$.

We now define $$\hat{h}(t)=h_{0}-\mu \int^t_{0} [\hat{U}_{y}(\tau,h_{0})+\rho\hat{V}_{y}(\tau,h_{0})]d\tau.$$
Then, $\hat{h}'(t)=-\mu (\hat{U}_{y}(t,h_{0})+\rho\hat{V}_{y}(t,h_{0}))\in C^{\frac{\theta}{2}}[0,T]$  and $\| \hat{h}'\|_{C^{\frac{\theta}{2}}}\leq C_{3},$  where $C_{3}$ depends on $\mu,$  $\rho,$  $h_{0},$  $\alpha,$  $\| U_{0}\|_{C^{2}[0,h_{0}]}$ and $\| V_{0}\|_{C^{2}[0,h_{0}]}$.

Now, we are ready to introduce the mapping $\Phi: (U,V,h)\rightarrow (\hat{U},\hat{V},\hat{h})$. We  claim that $\Phi$ maps $X_{T}$ into itself for sufficiently small $T$:\\
Indeed, if we take $T$ such that
 $$0< T \leq\min\left\{C_{1}^{\frac{-2}{1+\theta}},C_{2}^{\frac{-2}{1+\theta}},C_{3}^{\frac{-2}{\alpha}}\right\},$$
 we then have $$\|\hat{U}-U_{0}\|_{C(\mathcal{R})}\leq \| \hat{U}\|_{C^{0,\frac{1+\theta}{2}}(\mathcal{R})}T^{\frac{1+\theta}{2}}\leq C_{1}T^{\frac{1+\theta}{2}}\leq 1,$$
 $$\| \hat{V}-V_{0}\|_{C(\mathcal{R})}\leq \| \hat{V}\|_{C^{0,\frac{1+\theta}{2}}(\mathcal{R})}T^{\frac{1+\theta}{2}}\leq C_{2}T^{\frac{1+\theta}{2}}\leq 1,$$
 $$\| \hat{h}'-\tilde{h}\|_{C[0,T]}\leq \| \hat{h}'\|_{C^{\frac{\theta}{2}}[0,T]}T^{\frac{\theta}{2}}\leq C_{3}T^{\frac{\theta}{2}}\leq 1.$$
Thus we have $\Phi$  as a map from $X_{T}$ into itself.

Now we show that $\Phi$ is a contraction mapping for sufficiently small $T$. Let $(\hat{U}_{i},\hat{V}_{i},\hat{h}_{i})\in X_{T}$ for $i=1,2$. We set $\bar{U}=\hat{U_{1}}-\hat{U_{2}},$  and $\bar{V}=\hat{V_{1}}-\hat{V_{2}}.$  Then,
 \begin{equation}\label{5.5}
\left\{\begin{array}{l}
\ds{\frac{\partial \bar{U}}{\partial t}=A^{2}(h_{2}(t),y(t))\bar{U}_{yy}+[B(h_{2}(t),y(t))-C(h_{2}(t),y(t))]
\bar{U}_{y}+\mathbf{F},} \vspace{7 pt}\\
\text{ for }t>0\text{ and } 0<y<h_{0}.\vspace{7 pt}\\
\ds{\frac{\partial\bar{V}}{\partial t}= DA^{2}(h_{2}(t),y(t))\bar{V}_{yy}+(DB(h_{2}(t),y(t))-C(h_{2}(t),y(t)))\bar{V}_{y}+
\mathbf{G},} \vspace{7 pt}\\
\text{ for }t>0 \text{ and }  0<y<h_{0}.\vspace{7 pt}\\
\ds{\bar{U}_{y}(t,0)=\bar{V}_{y}(t,0)=0,}\quad t>0,\vspace{7 pt}\\
\ds{\bar{U}(t,h_{0})=\bar{V}(t,h_{0})=0,}\quad t>0,\vspace{7 pt}\\
\ds{\bar{U}(0,y)=\bar{V}(0,y)=0,} \quad 0\leq y\leq h_{0},
\end{array}\right.
\end{equation}
where
 \begin{equation*}
\begin{array}{lll}
\mathbf{F}:&=&[A^{2}(h_{1}(t),y(t))-A^{2}(h_{2}(t),y(t))]\hat{U}_{1yy}+[(B(h_{1}(t),y(t))-B(h_{2}(t),\vspace{7 pt}\\
&&y(t)))-
(C(h_{1}(t),y(t))-C(h_{2}(t),y(t))]\hat{U}_{1y}+F(U_{1},V_{1})-F(U_{2},V_{2}).
\end{array}
\end{equation*}
\begin{equation*}
\begin{array}{lll}
\mathbf{G}:&=&[DA^{2}(h_{1}(t),y(t))-DA^{2}(h_{2}(t),y(t))]\hat{V}_{1yy}+
[(DB(h_{1}(t),y(t))-DB(h_{2}(t),\vspace{7 pt}\\
&&y(t)))-(C(h_{1}(t),y(t))
-C(h_{2}(t),y(t)))]\hat{V}_{1y}+G(U_{1},V_{1})-G(U_{2},V_{2}).
\end{array}
\end{equation*}
Again, using standard $L^{p}$ estimates and the Sobolev embedding theorem, we have
$$\ds{\|\bar{U}\|_{C^{\frac{1+\theta}{2},1+\theta}(\mathcal{R})}\leq C_{4}(\| U_{1}-U_{2}\|_{C(\mathcal{R})}+\| V_{1}-V_{2}\|_{C(\mathcal{R})}+\| h_{1}-h_{2}\|_{C^{1}[0,T]}),}$$
$$\| \bar{V}\|_{C^{\frac{1+\theta}{2},1+\theta}}(D)\leq C_{5}(\| U_{1}-U_{2}\|_{C(\mathcal{R})}+\| V_{1}-V_{2}\|_{C(\mathcal{R})}+\| h_{1}-h_{2}\|_{C^{1}[0,T]}),$$
and
$$\| \bar{h}_{1}'-\bar{h}_{2}'\|_{C^{\frac{1+\theta}{2},1+\theta}}([0,T])\leq C_{6}(\| U_{1}-U_{2}\|_{C(\mathcal{R})}+\| V_{1}-V_{2}\|_{C(\mathcal{R})}+\| h_{1}-h_{2}\|_{C^{1}[0,T]}),$$
where the constants $C_{4},~C_{5},$ and $C_{6}>0$  depend on $A,$  $B,$  $C$ and $C_{i},$ for $ i=1, 2, 3$.

\noindent We also have
\begin{equation*}
\begin{array}{ll}
\|\bar{U}\|_{C(\mathcal{R})}+\|\bar{V}\|_{C(\mathcal{R})}+\|\bar{h}_{1}'-\bar{h}_{2}'\|_{C[0,T]}&\leq T^{\frac{1+\theta}{2}}\|\bar{U}\|_{C^{\frac{1+\theta}{2},1+\theta}(\mathcal{R})}+T^{\frac{1+\theta}{2}}\| \bar{V}\|_{C^{\frac{1+\theta}{2},1+\theta}(\mathcal{R})}\vspace{7 pt}\\
~&+T^{\frac{\theta}{2}}\|\bar{h}_{1}'-\bar{h}_{2}'\|_{C^{\frac{1+\theta}{2},1+\theta}([0,T])}.
\end{array}\end{equation*}

Based on the above, if $T\in (0,1],$  then
\begin{equation*}
\begin{array}{ll}
\|\bar{U}\|_{C(\mathcal{R})}+\|\bar{V}\|_{C(\mathcal{R})}+\| \bar{h}_{1}'-\bar{h}_{2}'\|_{C([0,T])}&\leq C_{7}T^{\frac{\theta}{2}}\left\{\|U\|_{C(\mathcal{R})}
+\|V\|_{C(\mathcal{R})}\right.\vspace{7 pt}\\
~&\left.+\| h_{1}'- h_{2}'\|_{C([0,T])}\right\},
\end{array}\end{equation*}
where $C_{7}:=\max \{C_{4},C_{5},C_{6}\}$. We choose
$$T=\frac{1}{2} \min \left\{1,\frac{h_{0}}{8(1+\tilde{h})},C_{1}^{\frac{-2}{1+\theta}}
,C_{2}^{\frac{-2}{1+\theta}},C_{3}^{\frac{-2}{\theta}},C_{7}^{\frac{-2}{\theta}} \right\},$$
and apply the contraction mapping theorem to conclude that $\Phi$ has a unique fixed point in $X_{T}$. This completes the proof of Lemma \ref{local.exist}.

 \end{proof}

  We now turn to the

  \begin{proof}[Proof of Lemma \ref{estimates}]
  The strong maximum principle yields that $u>0$ and  $\upsilon>0,$  for all $t\in [0,T]$ and $x\in [0,h(t)).$ Since $u(t,h(t))=\upsilon(t,h(t))=0,$   then Hopf  Lemma yields that $u_{x}(t,h(t))<0$ and $\upsilon_{x}(t,h(t))<0$ for all $t\in(0,T].$ Thus, $h'(t)>0$ for all $t\in (0,T]$.

Now, we consider the following initial value problem
\begin{equation}\label{5.9}
\bar{u}'(t)=\bar{u}(1-\bar{u})~ \text{ for } t>0,\qquad
\bar{u}(0)=\|u_{0}\|_{\infty} .
\end{equation}
The comparison principle implies that $u(t,x)\leq \bar{u}(t,x)\leq\max\{1,\|u_{0}\|_{\infty}\}$ for all $t\in[0,T]$ and for all  $x\in[0,h(t)]$. Similarly, we consider the following problem
\begin{equation}\label{5.10}
\bar{\upsilon}'(t)=\kappa\bar{\upsilon}(1-\frac{\bar{\upsilon}}{M_{1}+\alpha})\text{ for } t>0,\qquad
\bar{\upsilon}(0)=\|\upsilon_{0}\|_{\infty} ,
\end{equation}
to conclude, via the comparison principle, that $\upsilon(t,x)\leq \max\{M_{1}+\alpha,\|\upsilon_{0}\|_{\infty}\}$ for all $t\in[0,T]$  and $x\in[0,h(t)]$.

We turn now to prove that $h'(t)\leq \Lambda~\text{for}~t\in (0,T]$. In order to achieve this, we shall compare $u$ and $\upsilon$ to the following two auxiliary functions
\begin{equation*}\label{5.11}
\begin{array}{l}
\omega_{1}(t,x)=M_{1}[2M(h(t)-x)-M^{2}(h(t)-x)^{2}] \text{ for }t\in[0,T]~\&~ x\in[h(t)-M^{-1}, h(t)],\vspace{7 pt} \\

\text{and}\vspace{7 pt} \\
\omega_{2}(t,x)=M_{2}[2M(h(t)-x)-M^{2}(h(t)-x)^{2}] \text{ for }t\in[0,T]~\&~x\in[h(t)-M^{-1}, h(t)].
\end{array}
\end{equation*}
As a first choice, we pick $M=\ds{\max \left\{\frac{1}{h_{0}},\frac{\sqrt{2}}{2},\sqrt{\frac{\kappa}{2D}}\right\}}$  in order to obtain that
\begin{equation}\label{5.12}
\left\{\begin{array}{l}
\partial_{t}\omega_{1}-\partial_{xx}\omega_{1}\geq 2M_{1}M^{2}\geq u\geq u(1-u-\delta\upsilon)=\partial_{t}u-\partial_{xx}u,\vspace{7 pt}\\
\partial_{t}\omega_{2}-D\partial_{xx}\omega_{2}\geq 2DM_{2}M^{2}\geq \kappa\upsilon\geq\kappa\upsilon(1-\frac{\upsilon}{u+\alpha})=\partial_{t}\upsilon-D\partial_{xx}\upsilon,\vspace{7 pt}\\
\omega_{1}(t,h(t))=0=u(t,h(t)),\vspace{7 pt}\\
\omega_{2}(t,h(t))=0=\upsilon(t,h(t)),\vspace{7 pt}\\
\omega_{1}(t,h(t)-M^{-1})=M_{1}\geq u(t,h(t)-M^{-1}),\vspace{7 pt}\\
\omega_{2}(t,h(t)-M^{-1})=M_{2}\geq\upsilon(t,h(t)-M^{-1}).
\end{array}\right.
\end{equation}
We plan to use a comparison argument to complete the proof. For this, we need to have $\omega_{1}(0,x)\geq u_{0}(x)$ and $\omega_{2}(0,x)\geq \upsilon_{0}(x)$. Note that, for $x\in [h(t)-M^{-1}, h(t)],$ \[u_{0}(x)=-\int^{h_{0}}_{x}u'(s)ds\leq (h_{0}-x)\|u'\|_{C[0,h_{0}]},\] \[\upsilon_{0}(x)=-\int^{h_{0}}_{x}\upsilon'(s)ds\leq (h_{0}-x)\|\upsilon'\|_{C[0,h_{0}]},\]
 \[\omega_{1}(0,x)=M_{1}M(h_{0}-x)[2-M(h_{0}-x)]\geq M_{1}M(h_{0}-x)\] \[\text{and }\omega_{2}(0,x)=M_{2}M(h_{0}-x)[2-M(h_{0}-x)]\geq M_{1}M(h_{0}-x)\text{ for }x\in [h_{0}-M^{-1},h_{0}].\] Thus, if $\ds{ M =\max\left\{ \frac{\|u'\|_{C[0,h_{0}]}}{M_{1}}, \frac{\|\upsilon'\|_{C[0,h_{0}]}}{M_{2}}\right\}},$ then we have $\omega_{1}(0,x)\geq u(0,x)$ and $\omega_{2}(0,x)\geq \upsilon(0,x).$ By now, we have two constraints that $M$ should satisfy.  We choose $M$ such that  \[M = \max\left\{\frac{1}{h_{0}},~\frac{\sqrt{2}}{2},~\sqrt{\frac{\kappa}{2D}},~\frac{\|u'\|_{C[0,h_{0}]}}{M_{1}}, ~\frac{\|\upsilon'\|_{C[0,h_{0}]}}{M_{2}}\right\}.\] Then, the comparison principle yields that $\omega_{1}\geq u$ and  $\omega_{2}\geq \upsilon$ for $t\in [0,T]$ and $x\in [h(t)-M^{-1}, h(t)]$. Since $\omega_{1}(t,h(t))=u(t,h(t))=0$ and $\omega_{2}(t,h(t))=\upsilon(t,h(t))=0,$ we then obtain that
\[\partial_{x}u(t,h(t))\geq \partial_{x}\omega_{1}(t,h(t))=-2MM_{1} \text{ and } \partial_{x}\upsilon(t,h(t))\geq\partial_{x}\omega_{2}(t,h(t))=-2MM_{2}.\] Therefore, we have $h'(t)\leq \Lambda ,$ where $\Lambda:=2M\mu(M_{1}+ \rho M_{2})$. The proof of Lemma \ref{estimates} is now complete.
\end{proof}

\section{Discussion and summary of the results}\label{discuss}
In this paper, we considered a Leslie-Gower and Holling-type II predator-prey model in a one-dimensional environment. The model studies two species that initially occupy the region $[0,h_{0}]$  and both have a tendency to expand their territory. We obtain several results in  this setting.
\begin{enumerate}[(i)]
\item Theorem \ref{th3.1} and Theorem \ref{th3.2} provide the asymptotic behavior of the two species when spreading success and spreading failure, in terms of $h_{\infty}$:

 If $h_{\infty}=+\infty,$  then we have $$\lim_{t\rightarrow+\infty}u(t,x)=u^{*}, \lim_{t\rightarrow+\infty}\upsilon(t,x)=\upsilon^{*}.$$

  If $h_{\infty}<+\infty,$  then we have
$$\lim_{t\rightarrow+\infty}\|u(t,\cdot)\|_{C[0,h(t)]}=0,  \lim_{t\rightarrow+\infty}\| \upsilon(t,\cdot)\|_{C[0,h(t)]}=0.$$

\item A spreading-vanishing dichotomy can be established by using Lemma \ref{le3.1} and the critical length for the habitat can be characterize by $h_{*},$  in the sense that the two species will spread successfully if $h_{\infty}>h_{*},$  while the two species will vanish eventually if $h_{\infty}\leq h_{*}$. If the size of initial habitat $h_{0}$ is not less than $h_{*},$  or $h_{0}$ is less than $h_{*},$  but $\mu \geq \bar{\mu}$ or $0 < D\leq D^{*} ,$  then the two species will spread successfully. While if the size of initial habitat is less than $h_{*}$ and
$\mu \leq \underline{\mu}$ or $D^{*}< D \leq \kappa,$  then the two species will disappear eventually.\\

\item Finally, Theorem \ref{spreading.speed} reveals that the spreading speed (if exists) is between the minimal speed of traveling wavefront solutions for the predator-prey model on the whole real line (without a free boundary) and an elliptic problem induced from the original model.
\end{enumerate}


\begin{thebibliography}{99}

\bibitem{MM2003} M.A. Aziz-Alaoui and M. Daher-Okiye, \emph{Boundedness and Global Stability or a Predator-prey Model with Modified Leslie-Gower and Holling-Type II Schemes}, Applied Mathematics Letters  16 (2003), pp. 1069--1075.

\bibitem{Cantrell2003} R.S. Cantrell and C. Cosner, \emph{Spatial Ecology via Reaction-Diffusion Equations}, Wiley, Chichester 2003.

\bibitem{L-G4} F. Chen, L. Chen, and X. Xie, \emph{On a Leslie-Gower predator-preymodel incorporating a prey refuge}, Nonlinear Analysis: Real World Applications  10 (2009), pp. 2905--2908.

\bibitem{Chenxf} X.F. Chen and A. Friedman,  \emph{A free boundary problem arising in a model of wound healing}, SIAM J. Math. Anal. 32 (2000), pp. 778--800.

\bibitem{dushou}Y.H. Du and Z.G. Lin, \emph{Spreading - vanishing dichotomy in the diffusive logistic model with a free boundary},  SIAM J. Math. Anal. 42 (2010), pp. 377--405.

\bibitem{dusupandinf2014} Y.H. Du and Z.G. Lin, \emph{The diffusive competition model with a free boundary: invasion of a superior or inferior competitior}, Discrete Contin. Dyn. Syst. Ser. B 19 (2014), pp. 3105--3132.

\bibitem{b}  J.S. Guo and C.H. Wu, \emph{On a free boundary problem for a two-species weak competition system}, J.Dyn. Diff. Equat. 24 (2012), pp. 873--895.

\bibitem{HH1995}  S.B. Hsu and T.W. Huang, \emph{Global stability for a class of predator-prey systems}, 55 (1995), pp. 763--783.

\bibitem{L-G1} A. Korobeinikov, \emph{A Lyapunov function for Leslie-Gower predator-prey models}, Appl. Math. Lett. 14 (2001), pp. 697--699.

\bibitem{linzhigui2007} Z.G. Lin, \emph{A free boundary problem for a predator-prey model}, Nonlinearity  20 (2007), pp.1883--1892.

\bibitem{WN2016} W.J. Ni and M.X. Wang, \emph{Dynamics and patterns of a diffusive Leslie-Gower prey-predator model with strong Allee effect in prey}, J. Differential Equations 261(2016), pp. 4244--4274.

\bibitem{wangjie}  J. Wang, \emph{The selection for dispersal: a diffusive competition model with a free boundary}, Z. Angew. Math. Phys. 66 (2015), pp. 2143--2160.

\bibitem{w3} M.X. Wang, \emph{On some free boundary problems of the prey-predator model}, J. Differential Equations 256 (2014), pp. 3365--3394.

\bibitem{ziyou1} M.X. Wang, \emph{Spreading and vanishing in the diffusive prey-predator model with a free boundary}, Commun. Nonlinear Sci. Numer. Simul. 23 (2015), pp. 311--327.

\bibitem{ziyou2} M.X. Wang and Y. Zhang, \emph{Two kinds of free boundary problems for the diffusive prey-predator model}, Nonlinear Anal. Real World Appl. 24 (2015), pp. 73--82.

\bibitem{w2} M.X. Wang and J.F. Zhao, \emph{A free boundary problem for a predator-prey model with double free boundaries}, J. Dynam. Differential Equations (2015), pp. 1--23.

\bibitem{L-G5}  R.Z. Yang and J.J. Wei, \emph{The Effect of Delay on A Diffusive Predator-Prey System with Modified Leslie-Gower Functional Response}, Bull. Malays. Math. Sci. Soc. 40 (2017), pp. 51--73.

\bibitem{ziyou3} J.F. Zhao and M.X. Wang, \emph{A free boundary problem of a predator-prey model with higher dimension and heterogeneous environment} Nonlinear Anal. Real World Appl. 16 (2014), pp. 250--263.


\bibitem{wx1} Y. Zhang and M.X Wang, \emph{A free boundary problem of the ratio-dependent prey-predator model}, Applicable Analysis 94 (2015), pp. 2147--2167.

\bibitem{ZhouJun2014} J. Zhou,  \emph{Positive solutions of a diffusive Leslie-Gower predator-prey model with Bazykin functional response}, Z. Angew. Math. Phys. 65 (2014), pp. 1--18.

\bibitem{ziyou4} L. Zhou, S. Zhang and Z.H. Liu, \emph{A free boundary problem of a predator-prey model with advection in heterogeneous environment}, Appl. Math. Comput. 289 (2016), pp. 22--36.

\bibitem{zhouxiao} P. Zhou and D.M. Xiao, \emph{The diffusive logistic model with a free boundary in heterogeneous environment}, J.Differential Equations 256 (2014), pp. 1927--1954.

\end{thebibliography}
\end{document}